\documentclass[12pt,letterpaper]{article}
\usepackage{graphicx}
\usepackage{amsmath}
\usepackage{amsfonts}
\usepackage{amsthm}
\usepackage{amssymb}
\usepackage{color}

\usepackage{verbatim} 
\usepackage{url}
\usepackage{hyperref}

\newtheorem{theorem}{Theorem}[section]
\newtheorem{lemma}[theorem]{Lemma}

\newcommand{\be}{\begin{equation}}
\newcommand{\ee}{\end{equation}}
\newcommand{\bea}{\begin{eqnarray}}
\newcommand{\eea}{\end{eqnarray}}
\newcommand{\beas}{\begin{eqnarray*}}
\newcommand{\eeas}{\end{eqnarray*}}

\begin{document}

\title{Coins of Three Different Weights}

\author{Tanya Khovanova\\MIT \and Konstantin Knop\\Youth Math School of St.Petersburg State University}
\maketitle

\begin{abstract}
We discuss several coin-weighing problems in which coins are known to be of three different weights and only a balance scale can be used. We start with the task of sorting coins when the pans of the scale can fit only one coin. We prove that the optimal number of weighings for $n$ coins is $\lceil 3n/2\rceil -2$. When the pans have an unlimited capacity, we can sort the coins in $n+1$ weighings. We also discuss variations of this problem, when there is exactly one coin of the middle weight.
\end{abstract}

\section{Introduction}

We discuss several coin-weighing puzzles. Here is the setup for all of them.

\begin{quote}
There are $n$ coins that look the same. But they have different weights. The number of different possibilities for their weights is three. The pans are one of two types: tiny (allow only one coin)  or huge (allow any number of coins). Use the balance scale the smallest number of times for one of two goals: sort the coins into their weights or find a particular coin.
\end{quote}

There is a lot of literature in the case when there are two different weights. In this case usually one weight is designated as real and the other as counterfeit. We cannot possibly cover all available articles on the subject, so we just list some of them.

The first publication was by E. D. Schell in the January 1945 issue of the American Mathematical Monthly \cite{Schell}. It discussed a problem of finding one fake coin that is lighter than real coins out of 9 coins total. The problem of one fake coin out of 12 coins when it is not known whether the fake coin is heavier or lighter appeared almost at the same time \cite{Eves}. The solution generated a lot of publications \cite{Descartes, Goodstein, Grossman}.

The generalization of the latter puzzle for any number of coins is covered in \cite{Dyson, Fine}. There is also a variation where we need to find the fake coin, but do not need to tell whether is it heavier or lighter \cite{Dyson, Mauldon}.

In \cite{Pyber, AL} bounds close to $\log_3n$ were found for finding the fake coins if the number of fake coins is limited by a small fixed number. The case when we do not know the number of fake coins is covered by \cite{Hu, Purdy}.

A set of papers were written to answer the question of whether or not all the coins are of the same weight \cite{AKV, AK, AV, KV}. Another direction of this problem is when you are allowed to use several balance scales in parallel \cite{URTour, Knop, K}.

A survey of coin-weighing problems was written by R.~K.~Guy and R.~J.~No\-wa\-kowsky \cite{Guy} and a classification of different coin-weighing problems by E.~Purdy \cite{Purdy}.

We did not find papers discussing three different weights in a general setting. We only found some puzzles in the book by J.-C.~Baillif \cite{Baillif}.

Throughout the paper we define \textit{sorting} of the coins as partitioning of the coins into groups, where the coins in each group weigh the same, the coins in different groups have different weights, and for every two groups we know which one contains lighter coins. We start with general theory in Section~\ref{sec:gt}. Then we discuss particular problems:

\begin{itemize}
\item Section~\ref{sec:twoweights}: Pans are tiny, the number of possible weights is 2, the goal is to sort. We prove that the number of weighings is not more than $n-1$ and at least $\lceil n/2\rceil$
\item Section~\ref{sec:threeweights}: Pans are tiny, the number of possible weights is 3, the goal is to sort. We prove that sorting can be done in not more than $\lceil 3n/2 \rceil -2$ weighings, and it is impossible to guarantee fewer weighings.
\item Section~\ref{sec:largerpansmiddleweight}: Pans are huge, the number of possible weights is 3, the goal is to sort. Additional constraint: one coin of the middle weight. We prove that  the coins can be sorted in $n$ weighings.
\item Section~\ref{sec:findmiddle}: Pans are huge, the number of possible weights is 3. Additional constraint: one coin of the middle weight. The goal is to find the middle coin. We prove that  the middle-weight coin can be found in $\lceil n/2 \rceil + \log _3 n$ weighings.
\item Section~\ref{sec:threeweighthugepans}: Pans are huge, the number of possible weights is 3, the goal is to sort. We prove that  the coins can be sorted in $n+1$ weighings.
\end{itemize}

\section{General Theory}\label{sec:gt}

Suppose we have $n$ coins of three different weights $w_1$, $w_2$, and $w_3$, where $w_1 < w_2 < w_3$. The total number of different possibilities of assigning these weights to coins is $3^n$. We cannot always differentiate some of these possibilities using just a balance scale. For example, if all the coins are of the same weight, there is no way to decide whether this weight is $w_1$, $w_2$, or $w_3$. The number of different possibilities that are distinguishable by a balance scale is known to be $3^n - 2\cdot 2^n +2$; see sequence A101052 in the On-Line Encyclopedia of Integer Sequences (OEIS) \cite{OEIS}: 1, 3, 13, 51, 181, $\ldots$.

Each weighing divides the space of possibilities into three parts: the coins in both pans weigh the same, the coins in the left pan are lighter, and the coins in the left pan are heavier. For optimality we want these parts to be of equal or close to equal sizes. The following lower bound follows:

\begin{lemma}
For $n>2$, the number of needed weighings is at least $n$.
\end{lemma}

\begin{proof}
It is enough to see that $3^n - 2\cdot 2^n +2 > 3^{n-1}$, for $n > 2$.
\end{proof}

For $n=1$, we do not need any weighings. For $n=2$ comparing the two coins to each other is the only weighing that is needed.

For $n > 2$ we argue that the most profitable use of the scale for the first weighing is to compare one coin against one coin. For simplicity, we assume that the weights $w_1$, $w_2$ and $w_3$ are linearly independent over integers. That means if the scale balances with $k_1$ coins on the left pan and $k_2$ coins on the right pan, then $k_1=k_2$ and the partition of coins into different weights is the same for both pans.

Suppose we pick two coins, $a$ and $b$, that we did not weigh yet and compare them against each other. If $a$ is lighter than $b$, then the possibilities of the weight distributions for these two coins are: $(w_1,w_2)$, $(w_1,w_3)$, and $(w_2,w_3)$. Similarly, if $a$ is heavier than $b$, then the possibilities are $(w_2,w_1)$, $(w_3,w_1)$, and $(w_3,w_2)$. If $a$ and $b$ weigh the same, then we have $(w_1,w_1)$, $(w_2,w_2)$, and $(w_3,w_3)$. That means at the beginning, when we do not have any more information, comparing pairs of coins is effective: it divides the space of possibilities into three equal parts. We will show that using more coins in the first weighing divides the space less equally.

Let us see what happens if instead of comparing one coin to one coin, we compare two coins against two coins. By symmetry, the two unbalanced outcomes divide the space of possibilities equally. Now we count the number of possibilities if the pans are balanced. If the two coins on the left pan are of the same weight, then the total weight on each pan is $2w_1$, $2w_2$ or $2w_3$. For each of these total weights all four coins are uniquely determined. So these are 3 possibilities.

Now suppose the two coins on the left pan are of different weights, for a total of $w_1+w_2$, $w_1+w_3$, or $w_2+w_3$. Then the left and the right pan have the same distribution of coins, but we can interchange the two coins in each pan. Thus for every total weight we have 4 possibilities. So these are 12 possibilities.

The total number of possibilities for the weights when the pans are balanced is 15, which is less than 27: the desired third of all possibilities of weights of four coins. 

Thus comparing two coins against two coins is not as good a strategy as comparing one coin against one coin. Now we generalize this observation to more than two coins.

\begin{theorem}
If $k$ coins on each side balance, then there are $ \sum_{i \mathop = 0}^k \binom k i^2 \binom {2 i} i$ possibilities for the distribution of weights among these $2k$ coins.
\end{theorem}

\begin{proof}
For each partition of $k$ into three non-negative integers $k_1$, $k_2$, and $k_3$, where $k_1+k_2+k_3=k$, the number of possibilities is $(\frac{k!}{k_1! k_2! k_3!})^2 = \binom{k}{k_1}^2 \binom{k-k_1}{k_2}^2$. The formula for the sum of squares of binomial coefficients is well known \cite{Spiegel}: 

$$\sum_{i \mathop = 0}^m \binom m i^2 = \binom {2 m}{m}.$$ 

Thus, the total is 

\begin{multline*}
\sum_{k_1+k_2+k_3=k}\binom{k}{k_1}^2 \binom{k-k_1}{k_2}^2= \sum_{k_1=0}^k \binom{k}{k_1}^2 \left(\sum_{k_2=0}^{k-k_1} \binom{k-k_1}{k_2}^2\right) = \\
\sum_{k_1=0}^k \binom{k}{k_1}^2 \binom {2(k-k_1)} {(k-k_1)}=\sum_{i = 0}^k \binom k i^2 \binom {2 i} i.
\end{multline*}
\end{proof}

With three coins ($k=3$) in each pan, assuming linear independence of the three weights, we get 93 possibilities for balancing; for 4 coins we get 639 possibilities. This is the sequence A002893 in the OIES. Let $a_n$ denote the number of possibilities for the weights when we weigh $n$ coins against $n$ coins and they balance. The total number of possibilities for the weights of $2n$ coins is $3^{2n}$. Let $p_n = a_n/3^{2n}$ be the ratio. We already saw that $p_1=1/3 \approx 0.333$, $p_2 = 15/81 \approx 0.185$, $p_3 = 93/729 \approx 0.128$ and so on.

\begin{theorem}
The sequence $p_n$ decreases.
\end{theorem}

\begin{proof}
We can show that $a_{n+1} < 9a_n$. As $a_1=3$ and $a_2=15$, this is true for $n=1$. We proceed by induction using the recurrence for $a_n$ \cite{OEIS}: $(n+1)^2 a_{n+1} = (10n^2+10n+3) a_n - 9n^2  a_{n-1}$. Let $a_n < 9a_{n-1}$. Then $(n+1)^2  a_{n+1} = (10n^2+10n+3)  a_n - 9n^2  a_{n-1}  < (10n^2+10n+3)  a_n - n^2  a_n = (9n^2+10n+3) a_n < 9(n+1)^2a_n$. Thus, $a_{n+1} < 9a_n$.
From this, $p_{n+1} = a_{n+1}/9^{n+1} < a_n/9^n=p_n$.
\end{proof}

Since $p_1=1/3$, we see that $p_n < 1/3$, for $n > 1$. This result allows us to formulate the following rule.

\textbf{Rule.} It is a good idea to start weighings with comparing one coin against one coin. 

Our reasoning applies not only to the first weighing. If there are two coins that have not been weighed yet, then comparing one against the other does not contradict potential optimality.

\textbf{General Rule.} It is a good idea to compare two coins that have not been weighed yet.

Most coin-weighing problems assume that the capacity of each pan is unlimited. But we see that at least at the beginning we do not need this capacity; we just need to fit one coin.

We define \textit{huge} pans as pans with an unlimited capacity and \textit{tiny} pans as pans that only fit one coin. 

Because of the importance of tiny pans, we start our study with them. After that we use our findings and methods for huge pans.

\section{Tiny Pans. Two Weights. Sort}\label{sec:twoweights}

As a warm-up we will cover the sorting problem with tiny pans and two different weights. 

\begin{quote}
There are $n > 1$ coins that look the same. But they are either real of fake. Real coins all weigh the same. Fake coins weigh the same as each other, but they are lighter than real coins. The balance scale allows only one coin on each pan. Use the balance scale the smallest number of times to sort the coins by weight.
\end{quote}

\begin{theorem}\label{thm:twoweights}
The optimal number of required weighings is not more than $n-1$ and at least $\lceil n/2\rceil$.
\end{theorem}

\begin{proof}
Pick a coin and compare it to all other coins one by one. This strategy allows us to sort the coins in $n-1$ weighings. On the other hand, each coin needs to be put on the scale, and each weighing uses 2 coins. That means we need at least $\lceil n/2\rceil$ weighings.
\end{proof}

Suppose we have $n$ coins and the first weighing of two coins is unbalanced. In this case we sorted these two coins in one weighing, and the problem is reduced to $n-2$ coins. On the other hand, suppose the weighing is balanced. Then we can put aside one of the coins, but we have to compare the other coin to something else to figure out whether these are the heavy coins or the lighter ones. After the first weighing we reduced the problem to $n-1$ coins.

So the actual number of weighings depends on how many times we get an unbalanced weighing in the process. For example, if all coins weigh the same, then every weighing is balanced. In this case any algorithm will require $n-1$ weighings to sort the coins.

\section{Tiny Pans. Three Weights. Sort}\label{sec:threeweights}

Now let us go back to our main goal and consider three weights.

\begin{quote}
There are $n$ coins that look the same. But they have different weights. The number of different possibilities for their weights is three. The balance scale allows only one coin on each pan. Use the balance scale the smallest number of times to sort the coins.
\end{quote}

We will show that sorting can be done in $\lceil 3n/2 \rceil -2$ weighings, and it is impossible to guarantee fewer weighings. First let us produce a strategy that requires not more than $\lceil 3n/2 \rceil -2$ weighings.

\begin{theorem}\label{thm:tp3w}
There is a strategy that requires not more than $\lceil 3n/2 \rceil -2$ weighings.
\end{theorem}

\begin{proof}
The strategy consists of three rounds. Before the first round we mark all coins as $U$ (meaning unknown). 

We start by comparing two coins. If they are not balanced, we mark the lighter coin as $L$ and the heavier as $H$. If two coins balance, we put one of them aside, removing its mark, and return the other one to the pile of coins marked $U$. For every coin put aside we keep track of which coin it balanced with. 

We keep weighing pairs of coins marked $U$ until at most one coin is left, at which point the first round ends. After the end of the round we have:

\begin{itemize}
\item two piles of the same size: one from the lighter pan marked $L$ and the other from the heavier pan marked $H$,
\item maybe one extra coin marked $U$,
\item some coins set aside. The weight of each set-aside coin matches one of the marked coins.
\end{itemize}

Note that in the first round a balanced weighing is profitable. The set-aside coins do not need to participate in later rounds: Their fate is decided with one weighing per coin, which is better than our goal of 1.5 weighings per coin. Hence we can assume the worst, that all the weighings in the first round are unbalanced.

In the worst case we have two piles of size $\lfloor n/2 \rfloor$ at the end of the first round. The first pile has only lighter and middle coins, and the second pile has only heavier and middle coins. Also, we might have one extra coin we know nothing about.

In the second round we process the lighter pile and the heavier pile separately. By Theorem~\ref{thm:twoweights} we need not more than $\lfloor n/2 \rfloor - 1$ weighings to process the pile from the lighter pan and the same number of weighings to process the heavier pile.

As we are assuming that all the weighings are unbalanced in the first round, we can get the extra coin only if $n$ is odd. The third round is devoted to processing the extra coin. If by this time we have piles of three different weights, we can compare the last coin to a middle-weight coin and establish its weight. If by this time we have piles of only two different weights, we need to compare the extra coin to one coin from each of the other two piles and thus establish to which weight each pile belongs.

Summing up $\lfloor n/2 \rfloor$ for comparing unknown coins (the first round) and $2(\lfloor n/2 \rfloor - 1)$ for sorting coins in the heavier and lighter piles (the second round) and 2 for the extra coin if the number of coins is odd (third round), we get: $3n/2-2$ for even $n$ and $3\lfloor n/2 \rfloor$ for odd $n$. We can combine this to $\lceil 3n/2 \rceil -2$ for any $n$.
\end{proof}

Let is prove that this strategy is one of the optimal ones.

\begin{theorem}\label{thm:tp3wguarantee}
We cannot guarantee sorting in fewer than $\lceil 3n/2 \rceil -2$ weighings.
\end{theorem}

\begin{proof}
Consider an adversary who can control the outcome of each weighing and wishes to increase the number of weighings. Here is the adversary's strategy. At every point in the weighing process the adversary labels each coin by one of the letters $U$, $L$, and $H$.
\begin{itemize}
\item $U$ --- the coin did not participate in any weighing yet.
\item $L$ --- the coin participated in weighings and was lighter.
\item $H$ --- the coin participated in weighings and was heavier.
\end{itemize}

The adversary chooses the outcome of every weighing by the following rules:

\begin{enumerate}
\item $U < U$ --- Two $U$ coins are always unbalanced.
\item $U > L$ --- The $U$ coin is always heaver than an $L$ coin.
\item $U < H$ --- The $U$ coin is always lighter than an $H$ coin.
\item $L < H$ --- An $L$ coin is always lighter than an $H$ coin.
\item $L = L$ --- Two $L$ coins always balance.
\item $H = H$ --- Two $H$ coins always balance.
\end{enumerate}

The rules are chosen to maximize the number of weighings. For example, if our strategy requires comparing two $U$ coins, then it is not profitable for the adversary to balance them, so the adversary makes one of them, say left, be lighter than the other.

This adversary's strategy guarantees that the labeling is consistent. A coin in group $L$ is never heavier than another coin and similarly for group $H$.

Let us define the \textit{equality} graph of the group marked $L$. Each vertex of the graph corresponds to a coin labeled $L$. The vertices are connected if there was a weighing where the corresponding two coins were of equal weight. The number of connected components symbolizes how many coins in group $L$ we need to sort. Similarly, we define the equality graph for the $H$ group.

The adversary keeps track of three numbers: 
\begin{itemize}
\item $u$ --- the total number of coins marked as $U$.
\item $l$ --- the number of connected components in the equality graph of the group labeled $L$.
\item $h$ --- the number of connected components in the equality graph of the group labeled $H$.
\end{itemize}

At the beginning all coins are of type $U$, and the corresponding triple of numbers $(u,l,h)$ is $(n,0,0)$.

After the $U<U$ weighing the lighter coin will be marked as $L$ but it has not been compared to any other $L$ coins, so the number of connected components in $L$ increases by 1. Suppose before the weighing the numbers were $(u,l,h)$; then after the weighing they become $(u-2,l+1,h+1)$. Similarly we can process what happens to the triple of numbers $(u,l,h)$ after a weighing of each of the six types above. All the results are combined in Table~\ref{tab:results}.

\begin{table}[htbp]
\centering
  \begin{tabular}{| c | c | c |}
\hline
Type & Weighing & Result \\
    \hline
    1 & $U<U$ &  $(u-2,l+1,h+1)$  \\ 
    2 & $U>L$ &  $(u-1,l,h+1)$  \\ 
    3 & $U<H$ &  $(u-1,l+1,h)$  \\ 
    4 & $L<H$ &  $(u,l,h)$  \\ 
    5 & $L=L$ &  $(u,l-1,h)$ or $(u,l,h)$ \\ 
    6 & $H=H$ &  $(u,l,h-1)$ or $(u,l,h)$ \\
    \hline
  \end{tabular}
\caption{Results of different types of weighings}
\label{tab:results}
\end{table}

Note that if the strategy requires us to compare two $L$ coins, then if they were from the same connected component, the number of components does not change. If they were from different components, the number of components decreases by 1. Similarly for two $H$ coins.

Consider the value $s=1.5u+l+h$. At the beginning we have only $U$ coins and thus $s=1.5n$. At the end we assume that the adversary stops fighting us when the $U$ group becomes empty and the groups $L$ and $H$ each are divided into two connected components. At this point $s=4$. Note that the only scenario when we cannot reduce the $L$ and $H$ equality graphs to two connected components is when we have exactly one $L$ and $H$ coin. In this case $s=2$.

Our weighings do not increase $s$. So to minimize the number of weighings, we need to maximize up the rate with which we decrease $s$.

Weighings of the first type decrease $s$ by 1. Weighings of type 5 or 6 decrease it by 1 or 0. Weighings of type 2 or 3 decrease it by 0.5. Weighings of type 4 are the least profitable for us: they do not change $s$ at all. The most profitable weighings are of type 1, 5 and 6, and the strategy should not weigh the coins that are in the same component of an $L$ or $H$ group. In any case, we cannot reduce $s$ faster than by 1 for each weighing. 

At the end, when the adversary stops controlling the weighing, if there are two connected components in the $L$ and $H$ group, we need one more weighing for each group to compare the components in that group. This is equivalent to saying that at the end we need $s=2$. 

Hence, the number of weighings is at least $1.5n-2$. As it is an integer, the number of weighings is at least $\lceil 3n/2 \rceil -2$.
\end{proof}

Note that the same method can be applied to any number of different weights, say $k$, to reduce the problem to a smaller number of weights. 
Each reduction happens in one round. After the first round the coins will be marked $L$, $H$ and $U$. In addition we will know that all the coins marked $U$ weigh the same, the coins marked $L$ are not of the heaviest weight, and the coins marked $H$ are not of the lightest weight. 

We start with comparing two coins. If they are not balanced, we mark the lighter coin as $L$ and the heavier as $H$. If two coins balance, we put one of them aside, and return the other one to the pile of unknown coins. For every coin put aside we keep track of which coin it balanced with. If there is an extra coin we can use one more weighing to assign it to either the $L$ or the $H$ pile. 

Note that we used $\lfloor n/2\rfloor + 1$ weighings, and the coins in the $L$($H$) pile cannot contain coins of the heaviest(lightest) weight. 

After this round we have reduced the problem to two piles each having at most $k-1$ different weights. In addition, we have some coins that are known to be equal to a coin in one of the piles.

In the second round we can compare coins within $L$ and within $H$ by using at most $\lfloor n/2\rfloor +2$ weighings. As a result we will get four new piles that we can combine into three piles:

\begin{itemize}
\item $LL$ does not contain coins of the two heaviest weights,
\item $LH$ and $HL$ do not contain coins of the heaviest or the lightest weights,
\item $HH$ does not contain coins of the two lightest weights.
\end{itemize}

The next round will require at most $\lfloor n/2\rfloor +3$ weighings. Continuing in this manner, after $k-2$ rounds we have divided our coins into $k-1$ piles each with coins of two different weights. Overall we use at most $(k-2)\lfloor n/2\rfloor +(k-1)(k-2)/2$ weighings. After that we need at most $n-(k-1)$ weighings to sort the piles of two different weights each. The final bound is $(k-2)\lfloor n/2\rfloor +(k-1)(k-2)/2 + n-(k-1) = k\lfloor n/2\rfloor +(n-2\lfloor n/2\rfloor)+ (k-1)(k-2)/2 - (k-1) \leq k\lfloor n/2\rfloor +1 + (k-1)(k-4)/2$.

Coin weighing with tiny pans is equivalent to sorting algorithms. Minimizing the number of times the scale is used in coin problems is equivalent to minimizing the number of times the comparison operation is used in sorting algorithms. The main difference is that when sorting, the standard assumption is that the elements of the list are all different. Traditionally in coin-weighing problems, the number of possible values for weights is very small. This extra assumption creates room for improving on known sorting algorithms. The algorithm we described above is a coin-weighing analogue of the Tournament sort \cite{wikiTS}.

\section{Huge pans. Three Weights. One Coin of the Middle Weight. Sort}\label{sec:largerpansmiddleweight}

Suppose we have three weights, huge pans and one coin of the middle weight.

\begin{quote}
There are $n$ coins that look the same. But they have different weights. There are only three possibilities for their weights and one coin of the middle weight. The balance scale allows any number of coins on each pan. Sort the coins.
\end{quote}

First we solve the case when we have some extra information: For every coin we know that it belongs to one of two classes:

\begin{itemize}
\item $L$: the coin is either lighter or middle weight.
\item $H$: the coin is either heavier or middle weight.
\end{itemize}

In this case we can find the middle coin in $\lceil \log _3 n \rceil$ weighings.

\begin{lemma}\label{thm:lightheavypiles}
There are $n>2$ coins that look the same. But they have different weights. There are only three possibilities for their weights. There is at least one coin of each weight and exactly one coin of the middle weight. For every coin, it is known which one of the two extreme weights it does not weigh. Then the coins can be sorted in $\lceil \log _3 n \rceil$ weighings.
\end{lemma}

\begin{proof}
Suppose $3^k < n \leq 3^{k+1}$. Divide all the coins into three piles $A_1$, $A_2$, and $A_3$, so that $|A_1| = |A_2|$ and $3^k \geq \max \{|A_1|,|A_2|,|A_3|\}$, where $|A_i|$ is the size of the pile $A_i$. In addition, make sure that the number of $L$ and $H$ coins is the same in piles $A_1$ and $A_2$.

Now weigh $A_1$ against $A_2$. If the weighing balances, then the middle coin is in $A_3$. If $A_1$ is lighter, then the middle coin is either one of the coins marked $L$ in $A_2$, or one of the coins marked $H$ in $A_1$. In any case, we are in the same situation as before, but the size of the pile that might contain the middle coin is reduced to not more than $3^k$.

At the very end we will have a pile of size not more than 3 containing the middle-weight coin. If the size of the pile is 3, then comparing two coins marked the same will reveal the middle coin. If the size of the pile is less than 3 we can add more coins to the pile to make it 3 and, as before, compare two coins marked the same.

In this setup finding the middle coin is equivalent to sorting.
\end{proof}

Now we can get back to the original problem.

\begin{theorem}
If there is one middle coin, then the coins can be sorted in $n$ weighings.
\end{theorem}

\begin{proof}
We start with the same method as in Theorem~\ref{thm:tp3w}. We compare two coins of unknown weights. If two coins are unbalanced we sort them into two classes: $L$ and $H$. If two coins balance we put one aside and the other back into the pile of unknown coins. We continue this procedure until not more than one unknown coin is left. If the $L$ and $H$ classes each contain $k$ coins we performed $k+n-2k-1=n-k-1$ weighings. We compare the unknown coin, if any, to one of the coins from the $L$ or $H$ class and assign the class value to it. 

After not more than $n-k$ weighings we have not more than $2k+1$ coins marked $L$ and $H$ that need to be further processed, and they satisfy the assumptions of Lemma~\ref{thm:lightheavypiles}. Thus we can process them in $\lceil \log _3 (2k+1) \rceil \leq k$ weighings.
\end{proof}

Note that the same method can be applied to a problem with three weights where the number of coins of the middle weight is small. First we use the method described in Section~\ref{sec:threeweights} and not more than $n$ weighings to separate coins into two piles so that each pile has coins of two possible weights. Then we use one of the known methods of finding the fake coins from a set of coins when the number of fake coins is small \cite{Pyber, AL}.

\section{Find the coin of middle weight}\label{sec:findmiddle}

In this variation instead of sorting the coins we need to find the middle-weight coin.

\begin{quote}
There are $n$ coins that look the same. But they have different weights. There are only three possibilities for their weights and exactly one coin of the middle weight. The balance scale allows any number of coins on each pan. Find the middle coin.
\end{quote}

\begin{theorem}\label{thm:findmiddleweight}
The problem of finding the only middle coin can be solved in $\lceil n/2 \rceil + \lceil \log _3 n \rceil$ weighings.
\end{theorem}

\begin{proof}
We start by using the method above of comparing pairs of two unknown coins. We can discard pairs of coins of equal weight as they cannot contain the middle coin. After $\lfloor n/2 \rfloor$ weighings we have a pile $L$ from the lighter pan that can contain only the lighter coins or the middle coin, a pile $H$ from the heavier pan that can contain only the heavier coins or the middle coin, and possibly a leftover coin, which could be anything. If there is a leftover coin we can use one weighing to assign it either to $L$ or to $H$. Up to this point we used $\lceil n/2 \rceil$ weighings.

Now we can apply Lemma~\ref{thm:lightheavypiles} to find the middle coin out of the union of $L$ and $H$ coins in not more than $\lceil \log _3 n \rceil$ weighings.
\end{proof}

This puzzle variation was inspired by a problem at the 2013 MIT Mystery Hunt \cite{TK}. Every year MIT runs a mystery hunt, where many teams compete in solving puzzles and finding a coin at the end. During the 2013 MIT Mystery Hunt the team designing the hunt, very appropriately, designed the last puzzle as a coin-weighing puzzle:

\begin{quote}
There are $9$ coins that look the same and only one of them is real. But they have different weights. There are only three possibilities for their weights. Four coins are heavier and four coins are lighter than the real coin. The balance scale allows any number of coins on each pan. Find the real coin in not more than seven weighings.
\end{quote}

Theorem~\ref{thm:findmiddleweight} proves that the puzzle can be solved in seven weighings and shows how to do it. Actually it is possible to find the real coin in six weighings. This is due to the fact that the number of coins of each weight is known and the number of lighter and heavier coins is even. We leave it to the reader to find the mysterious place to save on one weighing.

\section{Huge pans. Three Weights. Sort}\label{sec:threeweighthugepans}

Let us revisit the case of sorting three weights, but now with huge pans.

\begin{quote}
There are $n$ coins that look the same. But they have different weights. The number of different possibilities for their weights is three. The balance scale allows any number of coins on each pan. Use the balance scale the smallest number of times to sort the coins.
\end{quote}

\begin{theorem}
There is a strategy that requires no more than $n+1$ weighings.
\end{theorem}

\begin{proof}
The strategy consists of four rounds. The first round is similar to the first round in Theorem~\ref{thm:tp3w}. Before it we mark all coins as $U$.

\textit{Round 1.} 
We start by comparing two coins. If they are not balanced, we mark the lighter coin as $L$ and the heavier as $H$ and we keep track of this unbalanced pair. If two coins balance, we put one of them aside, removing its mark, and return the other one to the pile of coins marked $U$. For every coin put aside we keep track of which coin it balanced with. We keep weighing pairs of coins marked $U$ until at most one coin is left, at which point the first round ends. 

Every balanced weighing reduces the number of coins we need to sort by one and uses one weighing. After this round we will need to process only marked coins (all but the set-aside coins).

At the end of this round we will have $k_1$ marked coins. Out of them $\lfloor k_1/2 \rfloor$ are marked $L$ (lighter or middle weight), $\lfloor k_1/2 \rfloor$ are marked $H$ (heavier or middle weight), and there may be an extra coin marked $U$. We have $k_1+1- \lfloor k_1/2 \rfloor$ more weighings we are allowed to use.

\textit{Round 2. Compare unbalanced pairs against each other until the two pans do not balance.} 
Suppose $(c_1,d_1)$ is the first unbalanced pair, and $(c_2,d_2)$ is the second unbalanced pair. We denote the weight of a coin with the same letter as a coin. Suppose $c_1 < d_1$ and $c_2 < d_2$. Each unbalanced pair weighs $w_1+w_2$, $w_1+w_3$ or $w_2+w_3$. As $w_1+w_2 < w_1 + w_3 < w_2+w_3$, comparing pairs give a lot of information. Suppose the pair $(c_1,d_1)$ balances against $(c_2,d_2)$, then $c_1=c_2$ and $d_1=d_2$. That means that two coins $c_2$ and $d_2$ can be put aside, removing their marks as each of them has a match. We used a total of two weighings for these two put-aside coins, one in round 1 and one in round 2.

If all pairs balance and there is no extra coin, then all the coins are sorted in $n-1$ weighings. If there is an extra coin, we can use two weighings to compare it to an $L$-coin and an $H$-coin and finish the sorting. In this case we need $n+1$ weighings. 

Now we can assume that we found two pairs of coins that do not balance. By this time we have $k_2$ marked coins that we still need to process, and we have $k_2- \lfloor k_2/2 \rfloor$ more weighings we are allowed to use. 

Suppose $c_1+d_1 < c_2+d_2$ are the four coins that form two unbalanced pairs. That means $c_1=w_1$ and $d_2=w_3$. We can unmark these coins as we know their weights. Now $c_1$ and $d_2$ becomes our \textit{reference} pair. 

\textit{Round 3. Compare the remaining pairs against the reference pair.} 
We compare the remaining unbalanced pairs against the reference pair. After such a weighing we know the weights of the coins in the non-reference pair, so we can remove marks from them. 

After we process all pairs we will have some coins put aside that match other coins, some coins of known weight and in addition we will have the coins $d_1$ and $c_2$, and maybe an extra coin. We only need to process $d_1$, $c_2$ and the extra coin to assign all the weights. If there is no extra coin, we have two more weighings we can use. If there is an extra coin, we have three more weighings we can use.

We know that $d_1$ is classified as $H$, and $c_2$ as $L$. So we use one more weighing to compare $d_1$ against $d_2$. The latter is of weight $w_3$, so at the end of this weighing we will know the weight of $d_1$. Similarly we use one more weighing to find the weight of $c_2$ by comparing it against $c_1$.

\textit{Round 4. Compare the extra coin to a coin of the middle weight.} 

If there is an extra coin, we have one more weighing we are allowed to use. We can determine the weight of the extra coin in one weighing by comparing it to a middle-weight coin.

Notice that two unbalanced pairs guarantee us the existence of this middle-weight coin. Indeed, the pairs $(c_1,d_1)$ and $(c_2,d_2)$ cannot both be of weight $w_1+w_3$. That means at least one of the coins $c_2$ or $d_1$ must be of the middle weight $w_2$.
\end{proof}

\section{Acknowledgements}

We are grateful to Julie Sussman, P.P.A., for thoroughly reading our drafts and helping us clarify our presentation.

\bigskip
\hrule
\bigskip

\noindent 2000 {\it Mathematics Subject Classification}: Primary 00A08 (Recreational mathematics); Secondary 68P10 (Searching and sorting), 05C85 (Graph algorithms), 94C12 (Fault detection; testing).

\noindent \emph{Keywords: } weighing, coins, sort, puzzle.

\end{document}